\documentclass{amsart}

\newtheorem{theorem}{Theorem}
\newtheorem{lemma}{Lemma}
\newtheorem{corollary}{Corollary}

\newcommand{\Z}{\ensuremath{\mathbf Z}}

\newcommand{\mcg}{\ensuremath{ \mathcal G}}

\newcommand{\beq}{\begin{equation}}
\newcommand{\eeq}{\end{equation}}
\newcommand{\benum}{\begin{enumerate}}
\newcommand{\eenum}{\end{enumerate}}

\usepackage{amsmath,amssymb,amsthm}

\title[Integers with missing digits]{Curious convergent series of integers with missing digits}
\author{Melvyn B. Nathanson}
\address{Lehman College (CUNY), Bronx, NY 10468}
\email{melvyn.nathanson@lehman.cuny.edu}



\dedicatory{To Ron Graham}

\date{\today}

\begin{document}

\begin{abstract}
A classical theorem of Kempner states that the sum of the reciprocals of positive integers 
with missing decimal digits converges.  
This result is extended to much larger families of ``missing digits'' sets of positive integers  
with both convergent and divergent harmonic series.  
\end{abstract}

\maketitle

\section{Kempner's theorem}

\begin{quotation}
``It is well known that the series 
\[
\sum_{n=1}^{\infty} \frac{1}{n} = \frac{1}{ 1 } +  \frac{1}{ 2 } +  \frac{1}{ 3 } + \cdots 
\]
diverges.  The object of this Note is to prove that if the denominators do not include all natural 
numbers 1, 2, 3, \ldots, but only those which do not contain any figure 9, the series converges.  
The method of proof holds unchanged if, instead of 9, any other figure 1, 2, \ldots, 8 is excluded, 
but not for the figure 0.''
\begin{flushright}
A. J. Kempner, \emph{Amer. Math. Monthly} 21 (1914), 48--50
\end{flushright}
\end{quotation}

A \emph{harmonic series} is a series of the form $\sum_{a\in A} 1/a$, where $A$ is a set of positive integers.  
Mathematicians have long been interested in the 
convergence or divergence of harmonic series. 
Let $c \in \{1,2,\ldots, 9\}$, and let $A_{10}(c)$ be the set of positive integers 
in which the digit $c$ does not occur in the usual decimal representation.  
Kempner~\cite{kemp14}  proved in 1914 that $\sum_{a\in A_{10}(c)} 1/a$ converges.  
He called this  ``a curious convergent series.''  
More generally, for every integer $g \geq 2$,  every positive integer $n$ 
has a unique  \textit{$g$-adic representation} of the form $n = \sum_{i=0}^k c_i g^i$, 
with digits $c_i \in \{0,1,2,\ldots, g-1\}$ for $i = 0,1,\ldots, k$ and $c_k \neq 0$. 
If $A_g(c)$ is the set of integers whose $g$-adic representation contains no digit $c$, then 
the infinite series $\sum_{a\in A_g(c)} 1/a$ converges.   
This includes the case $c=0$, which was not discussed by Kempner.

Kempner's theorem has been studied and extended by 
Baillie~\cite{bail79}, Farhi~\cite{fahr08}, Gordon~\cite{gord19}, Irwin~\cite{irwi16}, 
Lubeck-Ponomarenko~\cite{lube-pono18}, 
Schmelzer and Baillie~\cite{schm-bail08}, and Wadhwa~\cite{wadh75,wadh78}.  
It is Theorem 144 in Hardy and Wright~\cite{hard-wrig08}.

The $g$-adic representation is a special case of a more general method to represent the positive integers.  
A \emph{\mcg-adic sequence} is a strictly increasing sequence of positive integers $\mcg = (g_i)_{i=0}^{\infty}$ 
such that  $g_ 0 = 1$ and $g_{i}$ divides $g_{i+1}$ for all $ i \geq 0$.   
The integer quotients
\[
d_{i} = \frac{g_{i+1}}{g_{i}}
\]
satisfy $d_i \geq 2$ and 
\beq               \label{MissingDigits:gk-product}
g_{k+1} = g_{k}d_{k} = d_0 d_1 d_2 \cdots d_{k}
\eeq
for all $k \geq 0$.
Every positive integer $n$ has a unique representation in the form 
\beq               \label{MissingDigits:n}
n = \sum_{i=0}^{k} c_i g_i
\eeq
where $c_i \in \{0,1,\ldots, d_i-1\}$ for all $i  \in \{ 0,1,\ldots, k \}$ and $c_{k} \neq 0$. 
We call~\eqref{MissingDigits:n} the \textit{$\mcg$-adic representation} of $n$. 
This is equivalent to de Bruijn's additive system (Nathanson~\cite{nath2014-150, nath2017-172}).

Harmonic series constructed from sets of positive integers 
with missing $\mcg$-adic digits do not necessarily  converge. 
In Theorem~\ref{MissingDigits:theorem:converges} we construct sets of integers 
with missing $\mcg$-adic digits whose harmonic series converge, 
and also sets of integers 
with missing $\mcg$-adic digits whose harmonic series diverge.

\section{\mcg-adic representations with bounded quotients}

Define the \textit{interval of integers} 
\[
[a,b] = \{n \in \Z: a \leq n \leq b\}.
\]

Let $\mcg = (g_i)_{i=0}^{\infty}$ be a \mcg-adic sequence with quotients 
$d_i = g_{i+1}/g_i$. 
Let $I$ be a set of nonnegative integers, and, for all $i \in I$, 
let $U_i$ be a nonempty proper subset of  $[0,d_i-1]$.  
For every nonnegative integer $k$, let $A_{k}$ be the set of integers $n \in [g_k,g_{k+1} - 1]$ 
 whose  $\mcg$-adic representation 
 $n = \sum_{i=0}^k c_i g_i $
satisfies the following \textit{missing digits condition}: 
\beq                    \label{Digits:MissingDigitsCondition}  
c_i \in [0,d_i -1]  \setminus U_i \qquad \text{for all $i \in I \cap [0,k]$.}
\eeq

\begin{lemma}                      \label{MissingDigits:lemma:Ak} 
The set $A_k$ satisfies:
\benum
\item[(a)]
$A_k = \emptyset$ if and only if $k \in I$ and $U_k = [1,d_k-1]$.  
\item[(b)]
If $A_k \neq \emptyset$, then 
\beq               \label{MissingDigits:|Ak|}    
|A_k| \leq \prod_{\substack{i=0 \\ i \in I}}^{k} (d_{i}-|U_i|) 
 \prod_{\substack{i=0 \\ i \notin I}}^{k} d_{i} \leq 2|A_k|.
\eeq
\eenum
\end{lemma}

\begin{proof}  
We use the inequality 
\beq                                             \label{MissingDigits:inequality-2}  
x \leq 2(x-1) \qquad \text{for $x \geq 2$.}
\eeq

If $n \in [g_k,g_{k+1} - 1]$, then $n$ has the \mcg-adic representation
\[
n = \sum_{i=0}^{k-1} c_i g_i + c_kg_k 
\]
with $c_k \neq 0$ and so $c_k \in [1,d_k-1]$.
It follows that $A_k = \emptyset$  
if and only if $k \in I$ and $U_k = [1,d_k-1]$.

For $A_k \neq \emptyset$, there are three cases.
\begin{enumerate}
\item[(i)]
If  $k \in I$ and $0 \in U_k$, then  
\[
|A_k| = \prod_{\substack{i=0 \\ i \in I}}^{k} (d_{i}-|U_i|)  \prod_{\substack{i=0 \\ i \notin I}}^{k} d_{i} 
< 2|A_k|.
\]

\item[(ii)]
If $k \in I$ and and $0 \notin U_k$, then $U_k$ is a proper subset of $[1, d_k-1]$ 
and so $|U_k| \leq d_k-2$.  
Inequality~\eqref{MissingDigits:inequality-2} gives 
\[
d_k - |U_k| \leq 2\left( d_{k} - |U_k| -1 \right).
\]
We obtain 
\begin{align*}
|A_k| & = (d_{k} - |U_k| -1) \prod_{\substack{i=0 \\ i \in I}}^{k-1} (d_{i}-|U_i|) 
 \prod_{\substack{i=0 \\ i \notin I}}^{k} d_{i} \\
&  < \prod_{\substack{i=0 \\ i \in I}}^{k} (d_{i}-|U_i|)  \prod_{\substack{i=0 \\ i \notin I}}^{k} d_{i} \\
& \leq 2 (d_{k} - |U_k| -1) \prod_{\substack{i=0 \\ i \in I}}^{k-1} (d_{i}-|U_i|)  
\prod_{\substack{i=0 \\ i \notin I}}^{k} d_{i}  \\ 
& = 2|A_k|.  
\end{align*}

\item[(iii)]
We have $d_k \geq 2$ and so $d_k \leq 2(d_k-1)$ from inequality~\eqref{MissingDigits:inequality-2}.  
If $k \notin I$, then 
\begin{align*}
|A_k| 
& = (d_{k} - 1) \prod_{\substack{i=0 \\ i \in I}}^{k} (d_{i}-|U_i|)  \prod_{\substack{i=0 \\ i \notin I}}^{k-1} d_{i} 
 < \prod_{\substack{i=0 \\ i \in I}}^{k} (d_{i}-|U_i|)  \prod_{\substack{i=0 \\ i \notin I}}^{k} d_{i} \\
& =  d_{k}  \prod_{\substack{i=0 \\ i \in I}}^{k} (d_{i}-|U_i|)  \prod_{\substack{i=0 \\ i \notin I}}^{k-1} d_{i} 
\leq 2 (d_{k} - 1) \prod_{\substack{i=0 \\ i \in I}}^{k} (d_{i}-|U_i|)  \prod_{\substack{i=0 \\ i \notin I}}^{k-1} d_{i} \\
& = 2|A_k|.
\end{align*}
\end{enumerate} 
This completes the proof.  
\end{proof}

Let $A$ be a set of nonnegative integers, and let 
$A(n)$ be the number of elements $a \in A$ with $a \leq n$.  
The   \textit{upper asymptotic density} of the set $A$ is $d_U(A) = \limsup_{n \rightarrow \infty} A(n)/n$.
If  $\lim_{n \rightarrow \infty} A(n)/n $ exists, then  $d(A) = \lim_{n \rightarrow \infty} A(n)/n$ 
is called the \textit{asymptotic density} of the set $A$.  
The set $A$ has \emph{asymptotic density zero} if $d(A) = d_U(A) = 0$.

\begin{lemma}                                             \label{MissingDigits:lemma:series}
Let $A$ be a set of positive integers.  If $\sum_{a\in A} 1/a < \infty$, then $d(A) = 0$.
\end{lemma}

\begin{proof}
We show that $d_U(A) > 0$ implies $\sum_{a\in A} 1/a = \infty$.

If $d_U(A) = \limsup_{n \rightarrow \infty} A(n)/n = \alpha > 0$, 
then, for every $\varepsilon > 0$, we have 
\[
\frac{A(n)}{n} < \alpha + \varepsilon \qquad \text{for all integers $n \geq N(\varepsilon)$} 
\] 
and 
\beq                                          \label{MissingDigits:beta} 
\frac{A(n_i)}{n_i} > \alpha - \varepsilon \qquad  \text{for infinitely many integers $n_i $.} 
\eeq
Let $\varepsilon < \alpha/3$.  There is a sequence of positive integers 
$(n_i)_{i=0}^{\infty}$ satisfying inequality~\eqref{MissingDigits:beta} 
such that $n_0 \geq N(\varepsilon)$ and 
$n_{i} > 2n_{i-1}$ for all $i \geq 1$.  We have 
\begin{align*}
A(n_i) - A(n_{i-1}) & > (\alpha -\varepsilon) n_i - (\alpha +\varepsilon) n_{i-1} \\
& > (\alpha -\varepsilon) n_i - \frac{(\alpha +\varepsilon) n_{i}}{2} \\
& = n_i\left( \frac{\alpha - 3\varepsilon}{2}\right) 
\end{align*}
and so 
\begin{align*}
\sum_{\substack{a\in A \\ n_{i-1} < a \leq n_{i} }} \frac{1}{a} 
& \geq \frac{A(n_i) - A(n_{i-1})}{n_i} >  \frac{\alpha  - 3\varepsilon}{2} > 0.
\end{align*}
It follows that 
\[
\sum_{\substack{a\in A \\ 1 \leq a \leq n_k}} \frac{1}{a} 
\geq \sum_{i=1}^k \sum_{\substack{a\in A \\ n_{i-1} < a \leq n_{i} }} \frac{1}{a} 
> k\left(  \frac{\alpha  - 3\varepsilon}{2} \right) 
\]
and the infinite series $\sum_{a\in A} 1/a$ diverges. 
Equivalently, convergence of the infinite series $\sum_{a\in A} 1/a$ implies  $d(A) = 0$.  
This completes the proof.  
\end{proof}

The converse of Lemma~\ref{MissingDigits:lemma:series} is false.
The set of prime numbers has asymptotic density zero, but the sum of the reciprocals of the primes diverges.

\begin{theorem}                      \label{MissingDigits:theorem:converges} 
Let $\mcg = (g_i)_{i=0}^{\infty}$ be a \mcg-adic sequence with bounded quotients, that is,
\beq                \label{Digits:limsup}
d_i = \frac{g_{i+1}}{g_i}  \leq d 
\eeq
for some integer $d \geq 2$  and all $i =0,1,2,\ldots$.   
Let $I$ be a set of nonnegative integers, and, for all $i \in I$, 
let $U_i$ be a nonempty proper subset of  $[0,d_i-1]$.

Let $n = \sum_{i=0}^k c_i g_i$ be the $\mcg$-adic representation of the positive integer $n$.  
Let $A$ be the set of positive integers $n$  that satisfy the missing digits 
condition~\eqref{Digits:MissingDigitsCondition}. 
If 
\beq                \label{Digits:I(x)-1}
I(k) \geq \frac{(1+\delta)\log k}{\log (d/(d-1))}
\eeq 
for some $\delta > 0$ and all $k \geq k_0 = k_0(\delta) $, 
then the set $A$ has asymptotic density zero and the harmonic series $\sum_{a\in A} 1/a$ converges.

If 
\beq                \label{Digits:I(x)-2}
I(k) \leq \frac{(1-\delta) \log k}{\log d}
\eeq 
for some $\delta> 0$ and all $k \geq k_1 = k_1(\delta) $, 
then the harmonic series $\sum_{a\in A} 1/a$ diverges.
\end{theorem}

Kempner's theorem is the special case  $g_i = 10^i$, $d_i = 10$,  
and $U_i = \{9\}$ for all $i \in I = \mathbf{N}_0$.

\begin{proof}  
For all $k \in \mathbf{N}_0$,  the finite sets 
\[
A_{k} = A \cap [g_k,g_{k+1} - 1] 
\]
are pairwise disjoint and 
$ 
A = \bigcup_{k=0}^{\infty} A_k. 
$

For all $i \in I$, we have 
\[
1 \leq |U_i| \leq d_i -1
\]
and 
\[
 \frac{1}{d}  \leq \frac{1}{d_i}  \leq \frac{d_i-|U_i|}{d_i} = 1 -  \frac{|U_i|}{d_i} \leq 1 - \frac{1}{d} < 1. 
\]

Let $I(k)$ satisfy inequality~\eqref{Digits:I(x)-1}.  We obtain 
\beq                                                \label{MissingDigits:UpperBound}
\left(1 - \frac{1}{d} \right)^{I(k)} \leq 
\left( \frac{d-1}{d} \right)^{ \frac{(1+\delta)\log k}{\log (d/(d-1))}} = \frac{1}{k^{1+\delta}}.
\eeq
If $a \in A_k$, then $a \geq g_k = d_0d_1\cdots d_{k-1}$.  By Lemma~\ref{MissingDigits:lemma:Ak},  
\begin{align*}
\sum_{ \substack{a\in A \\ a \geq g_{k_0}}} \frac{1}{a} 
& =  \sum_{k= k_0 }^{\infty} \sum_{a\in A_k} \frac{1}{a}  
\leq \sum_{k= k_0}^{\infty} \frac{|A_k|}{g_{k}} \\
& \leq  \sum_{k= k_0}^{\infty} \frac{d_k}{\prod_{i=0}^{k} d_{i} } 
\prod_{\substack{i=0 \\ i \in I}}^{k} (d_{i}-|U_i|)   \prod_{\substack{i=0 \\ i \notin I}}^{k} d_{i} \\
& \leq d  \sum_{k= k_0}^{\infty} \prod_{\substack{i=0 \\ i \in I}}^{k} \frac{d_i -|U_i|}{d_i} \\
& \leq d  \sum_{k= k_0}^{\infty} \left(1 - \frac{1}{d} \right)^{I(k)} \\
& \leq d  \sum_{k= k_0}^{\infty} \frac{1}{k^{1+\delta}} < \infty.
\end{align*}
Thus, the harmonic series converges.  By Lemma~\ref{MissingDigits:lemma:series}, 
the set $A$ has  asymptotic density zero.

Let $I(k)$ satisfy inequality~\eqref{Digits:I(x)-2}.  We obtain 
\beq                                                \label{MissingDigits:LowerBound}
\left(\frac{1}{d} \right)^{I(k)} \geq 
\left( \frac{1}{d} \right)^{ \frac{(1-\delta)\log k}{\log d}} = \frac{1}{k^{1-\delta}}. 
\eeq
If $a \in A_k$, then $a < g_{k+1} = d_0d_1\cdots d_{k-1}d_k$.  
By Lemma~\ref{MissingDigits:lemma:Ak},  
\begin{align*}
\sum_{\substack{a\in A \\ a \geq g_{k_1}} }\frac{1}{a} 
& =  \sum_{k=k_1}^{\infty} \sum_{a\in A_k} \frac{1}{a}  
\geq \sum_{k=k_1}^{\infty} \frac{|A_k|}{g_{k+1}} \\
& \geq \frac{1}{2}  \sum_{k=k_1}^{\infty} \frac{1}{\prod_{i=0}^{k} d_{i} } 
\prod_{\substack{i=0 \\ i \in I}}^{k} (d_{i}-|U_i|)   \prod_{\substack{i=0 \\ i \notin I}}^{k} d_{i} \\
&  = \frac{1}{2}  \sum_{k=k_1}^{\infty} \prod_{\substack{i=0 \\ i \in I}}^{k} \frac{d_i -|U_i|}{d_i}  
\geq  \frac{1}{2}  \sum_{k=k_1}^{\infty} \prod_{\substack{i=0 \\ i \in I}}^{k} \frac{1}{d_i} \\
& \geq  \frac{1}{2}   \sum_{k=k_1}^{\infty} \left(\frac{1}{d} \right)^{I(k)} \\
& \geq  \frac{1}{2}  \sum_{k=k_1}^{\infty} \frac{1}{k^{1-\delta}} 
\end{align*}
and the harmonic series diverges.  
This completes the proof.  
\end{proof}

\begin{corollary} 
Let $I$ be a set of nonnegative integers, and let $(v_i)_{i \in I}$ be a sequence of 0s and 1s.  
Let $A$ be the set of integers $n$ such that, 
if $n \in \left[2^k, 2^{k+1}-1 \right]$ has the $2$-adic representation $n = \sum_{i=0}^k c_i 2^i$, 
then $c_i = v_i$ for all $i \in I \cap [0,k]$.  
If 
\[
I(k) \geq (1+\delta)\log_2 k
\] 
for some $\delta > 0$ and all $k \geq k_0(\delta) $, 
then the harmonic series $\sum_{a\in A} 1/a$ converges.
If
\[
I(k) \leq (1-\delta) \log_2 k
\]
for some $\delta> 0$ and all $k \geq k_1(\delta) $, 
then the harmonic series $\sum_{a\in A} 1/a$ diverges.
\end{corollary}

\begin{proof}
For all $i \in I$, let $u_i = 1-v_i$ and $U_i = \{u_i\}$.  
Apply Theorem~\ref{MissingDigits:theorem:converges}.  
\end{proof}

It is an open problem to determine the convergence or divergence of $\sum_{a\in A} 1/a$ 
if $I(k) \sim \log_2 k$.

\section{\mcg-adic representations with unbounded quotients}
Let $\mcg = (g_i)_{i=0}^{\infty}$ be a \mcg-adic sequence 
with quotients $d_i = g_{i+1}/g_i$.   
Let $I$ be an infinite  set of nonnegative integers, and,   
for all $i \in I$, let $U_i$ be a nonempty proper subset of  $[0,d_i-1]$.    
If the sequence  $\mcg = (g_i)_{i=0}^{\infty}$ has bounded quotients $d_i \leq d$,  
then 
\[
\frac{|U_i|}{d_i} \geq \frac{1}{d}
\]
for all $i \in I$ and the infinite series $\sum_{i\in I} \frac{|U_i|}{d_i}$ diverges.
Equivalently, the convergence of this series implies that \mcg\
has unbounded quotients.

Let $n = \sum_{i=0}^k c_i g_i$ be the $\mcg$-adic representation of the positive integer $n$.  
Let  $A$ be the set of positive integers whose  $\mcg$-adic representations 
satisfy the missing digits condition~\eqref{Digits:MissingDigitsCondition}. 
The missing digits set $A$ is finite if and only if $I$ is a cofinite set of nonnegative 
integers and $U_i = [1,d_i-1]$ for all sufficiently large $i$.  
The harmonic series of a finite set of positive integers converges.

Theorem~\ref{MissingDigits:theorem:converges} shows that harmonic series constructed 
from infinite sets of integers with missing $\mcg$-adic digits do not always converge.  
It follows from Theorem~\ref{MissingDigits:theorem:converges} that if 
\[
I(k) \geq (\log k)^{1+\delta} 
\] 
for some $\delta > 0$ and all $k \geq k_0(\delta) $, 
and if $\sum_{a \in A} 1/a$ diverges, 
then the sequence \mcg\ must have \emph{unbounded quotients}, that is, 
\[
\limsup d_i = \infty.
\]

Theorem~\ref{MissingDigits:theorem:diverges}  gives a sufficient condition 
for the divergence of harmonic series of sets of positive integers 
constructed from \mcg-adic sequences with unbounded  quotients.  
We use the following inequality, 
which is easily proved by induction:                  
If $0 \leq x_i < 1$ for $i=1,\ldots, n$, then 
\beq               \label{MissingDigits:InfiniteProduct}
\prod_{i=1}^n (1-x_i) \geq 1-\sum_{i=1}^n x_i.
\eeq

\begin{theorem}              \label{MissingDigits:theorem:diverges} 
Let $\mcg = (g_i)_{i=0}^{\infty}$ be a \mcg-adic sequence, 
and let $n = \sum_{i=0}^k c_i g_i$ be the $\mcg$-adic representation of the positive integer $n$.  
Let $I$ be an infinite  set of nonnegative integers, and,   
for all $i \in I$, let $U_i$ be a nonempty proper subset of  $[0,d_i-1]$.   
Let  $A$ be the set of positive integers whose  $\mcg$-adic representations 
satisfy the missing digits condition~\eqref{Digits:MissingDigitsCondition}.  
If the set $A$ is infinite and if 
\beq                                      \label{MissingDigits:ConvergentSeries} 
\sum_{i\in I} \frac{|U_i|}{d_i} < \infty 
\eeq
then  the sequence $\mcg = (g_i)_{i=0}^{\infty}$ 
has unbounded  quotients and the harmonic series $\sum_{a\in A} 1/a$ diverges.
\end{theorem}

For example, the ``missing digits'' set constructed from 
$\mcg = (g_i)_{i=0}^{\infty}$ with  $g_i = 2^{i(i+1)/2}$ and $d_i = 2^{i+1}$ 
and with $I = \mathbf{N}_0$ 
and $U_i = \{0\}$ for all $i \in I$ has a divergent harmonic series.

\begin{proof}
Because the infinite series~\eqref{MissingDigits:ConvergentSeries} converges, 
there is an integer $i_0 \in I$ such that 
\[
\sum_{\substack{i=i_0\\i \in I}}^{\infty} \ \frac{|U_i|}{d_i} < \frac{1}{2}.
\]
Inequality~\eqref{MissingDigits:InfiniteProduct} implies that, for all $k \in \mathbf{N}_0$,  
\[
\prod_{\substack{i=i_0\\i \in I}}^{k} \ \left( 1 - \frac{|U_i|}{d_i} \right) 
\geq 1 - \sum_{\substack{i=i_0\\i \in I}}^{k}  \frac{|U_i|}{d_i} > \frac{1}{2} 
\]
and so 
\begin{align*}
\prod_{\substack{i=0\\i \in I}}^{k} \left( 1 - \frac{|U_i|}{d_i} \right) 
& = \prod_{\substack{i=0\\i \in I}}^{i_0 -1} \left( 1 - \frac{|U_i|}{d_i} \right) 
 \prod_{\substack{i=i_0 \\i \in I}}^{k} \left( 1 - \frac{|U_i|}{d_i} \right) \\
& > \frac{1}{2} \prod_{\substack{i=0\\i \in I}}^{i_0 -1} \left( 1 - \frac{|U_i|}{d_i} \right) = \delta > 0. 
\end{align*}

Let $A_k = A \cap [g_k, g_{k+1}-1]$. The set $A$ is infinite if and only if 
$A_k \neq \emptyset$ for infinitely many $k$.  
Applying inequality~\eqref{MissingDigits:|Ak|}   of Lemma~\ref{MissingDigits:lemma:Ak}, 
we obtain 
\begin{align*}
\sum_{a\in A} \frac{1}{a} 
& =  \sum_{ \substack{k=0 \\ A_k \neq \emptyset} }^{\infty} \sum_{a\in A_k} \frac{1}{a}  
\geq \sum_{  \substack{k=0 \\ A_k \neq \emptyset}  }^{\infty} \frac{|A_k|}{g_{k+1}} \\
& \geq  \frac{1}{2} \sum_{ \substack{k=0 \\ A_k \neq \emptyset} }^{\infty} \frac{1}{\prod_{i=0}^{k} d_{i} } \prod_{\substack{i=0\\i\in I}}^{k} (d_i - |U_i|) 
\prod_{\substack{i=0\\ i \notin I}}^{k} d_i \\
& =  \frac{1}{2} \sum_{ \substack{k=0 \\ A_k \neq \emptyset} }^{\infty} 
\prod_{ \substack{ i=0 \\ i \in I }}^{k}  \left( 1 - \frac{|U_i|}{d_i} \right) 
\end{align*}
and so the harmonic series $\sum_{a\in A} \frac{1}{a}$ diverges.  
This completes the proof.  
\end{proof}

\def\cprime{$'$} \def\cprime{$'$} \def\cprime{$'$}
\providecommand{\bysame}{\leavevmode\hbox to3em{\hrulefill}\thinspace}
\providecommand{\MR}{\relax\ifhmode\unskip\space\fi MR }
\providecommand{\MRhref}[2]{%
  \href{http://www.ams.org/mathscinet-getitem?mr=#1}{#2}
}
\providecommand{\href}[2]{#2}

\end{document}